\documentclass[11pt,reqno]{amsart}

\usepackage{amsmath,amssymb,amsfonts}
\usepackage[left=30mm,right=30mm,top=30mm,bottom=30mm]{geometry}
\usepackage{hyperref}
\usepackage{cite}
\usepackage[foot]{amsaddr}
\usepackage{mathptmx}
\usepackage{graphics}
\usepackage{graphicx}
\usepackage{color}
\usepackage{wrapfig}

\usepackage[format=hang]{caption}

\newcommand{\E}{\mathcal{E}}
\newcommand{\V}{\mathcal{V}}

\newcommand{\Z}{\mathbb{Z}}
\newcommand{\N}{\mathbb{N}}
\newcommand{\R}{\mathbb{R}}
\newcommand{\C}{\mathbb{C}}

\renewcommand{\L}{\mathsf{L}^2}
\newcommand{\W}{\mathsf{H}}

\newcommand{\h}{\mathfrak{h}}
\renewcommand{\H}{\mathcal{H}}
\newcommand{\hb}{{\mathbf{h}}}
\newcommand{\Hb}{{\mathbf{H}}}

\newcommand{\A}{\widetilde{A}}
\newcommand{\B}{\widetilde{B}}

\newcommand{\dom}{\mathrm{dom}}

\newcommand{\la}{\langle}
\newcommand{\ra}{\rangle}

\newcommand{\eps}{\varepsilon}
\newcommand{\e}{_{\varepsilon}}
\newcommand{\al}{\alpha}
\newcommand{\be}{\beta}
\newcommand{\ga}{\gamma}

\renewcommand{\u}{\mathbf{u}}
\renewcommand{\v}{\mathbf{v}}

\newcommand{\x}{\mathbf{x}}

\renewcommand{\d}{\mathrm{d}}

\newcommand{\ds}{\displaystyle}

\newcommand{\wt}{\widetilde}

\newcommand{\cupl}{\bigcup\limits}
\newcommand{\suml}{\sum\limits}

\theoremstyle{plain}
\newtheorem{theorem}{Theorem}[section]
\newtheorem*{theorem*}{Theorem}
\newtheorem{lemma}[theorem]{Lemma}
\newtheorem*{lemma*}{Lemma}

\theoremstyle{remark}
\newtheorem{remark}{Remark}[section]
\newtheorem*{remark*}{Remark} 
\theoremstyle{definition}

\numberwithin{equation}{section}

\begin{document}
\title[Periodic quantum graphs with predefined spectral gaps]{Periodic quantum graphs with predefined spectral gaps}
\author[Andrii Khrabustovskyi]{Andrii Khrabustovskyi\,$^{1,2}$}

\address{$^1$ Institute of Applied Mathematics, Graz University of Technology,
Austria}
\address{$^2$ Department of Physics, Faculty of Science, University of Hradec Králové,
 Czech Republic}
\email{khrabustovskyi@math.tugraz.at, andrii.khrabustovskyi@uhk.cz}

\begin{abstract}
Let $\Gamma$ be an arbitrary $\mathbb{Z}^n$-periodic metric graph, which does not coincide with a line. We consider  the Hamiltonian $\mathcal{H}_\varepsilon$ on $\Gamma$ with the action $-\varepsilon^{-1}{\mathrm{d}^2/\mathrm{d} x^2}$ 
on its edges; here  $\varepsilon>0$ is a small parameter. Let $m\in\mathbb{N}$. We show that under a proper choice of vertex conditions 
the spectrum  $\sigma(\mathcal{H}^\varepsilon)$ of $\mathcal{H}^\varepsilon$ has at least $m$ gaps as $\varepsilon$ is small enough.
We demonstrate that the asymptotic  behavior  of these gaps and  the asymptotic    behavior of the bottom of $\sigma(\mathcal{H}^\varepsilon)$  as $\varepsilon\to 0$ can be completely controlled through a suitable choice of coupling constants standing in those vertex conditions. We also show how to ensure for fixed (small enough) $\varepsilon$ the precise coincidence of the left endpoints of the first $m$ spectral gaps with predefined numbers.
\end{abstract}


\maketitle

\thispagestyle{empty}

\section*{Introduction}

Traditionally the name \emph{quantum graph} refers to a pair
$(\Gamma,\mathcal{H})$, where $\Gamma$ is a network-shaped
structure of vertices connected by edges of certain positive lengths  (\emph{metric graph}) and
$\mathcal{H}$ is a second order self-adjoint differential operator
on $\Gamma$ (\emph{Hamiltonian}). Hamiltonians are determined by differential
operations on the edges and certain interface conditions at the
vertices. We refer to the
monograph \cite{BK13}  for a broad overview and an extensive
bibliography on this topic.

Quantum graphs arise naturally in mathematics, physics,
chemistry and engineering as  simplified  models of wave propagation in
quasi-one-dimensional systems looking like  narrow neighborhoods
of  graphs. 
{Typical applications include} quantum wires \cite{KS99,KS03}, photonic crystals \cite{KK99,KK02}, graphene and carbon nanostructures \cite{KL07,KP07}, 
quantum chaos \cite{KoSm97,KoSm99} 
and many other areas. For more details concerning origins of quantum graphs see \cite{Ku02} and \cite[Chapter~7]{BK13}. 

In various applications (for example, to aforementioned graphene and carbon nano-structures, and photonic crystals) periodic infinite graphs are studied. 
In what follows in order to simplify the presentation (but without any loss of 
generality) 
we assume that our graphs
are embedded into $\mathbb{R}^d$ for some $d\in\N$. 
An infinite metric graph $\Gamma\subset \mathbb{R}^d$  is said to be  \textit{$\mathbb{Z}^n$-periodic} ($n\le d$) 
if it invariant under translations through some
 linearly
independent vectors $\nu_1,\dots,\nu_n\in\mathbb{R}^d$.
The Hamiltonian $\H$ on a $\mathbb{Z}^n$-periodic metric graph $\Gamma$ is said to be 
periodic if it commutes with these translations. 
 
It is well-known that  the spectrum of a
periodic Hamiltonian on a periodic metric graph can be represented as a locally finite union of
compact intervals (\textit{spectral bands}). The bounded open interval
is called a \textit{gap} if it has an empty intersection with the
spectrum, but its ends belong to it. 
The band structure of the spectrum suggests that gaps may exist
in principle. In general, however, the presence of gaps is not guaranteed: two spectral bands may 
overlap, and then the corresponding gap disappears. 
For instance, if $\Gamma$ is a
rectangular lattice, $\H$ is defined by the operation $-{\d ^2/ \d x^2}$ on the edges and the standard Kirchhoff conditions at
the vertices, then   $\sigma(\mathcal{H})$  has no gaps -- it coincides with $[0,\infty)$. 

Existence and locations of spectral
gaps are of primary interest because of
various applications, for example in physics of \emph{photonic crystals} -- periodic nanostructures, whose characteristic property is that the light waves at certain optical frequencies fail to propagate in them, which is caused by gaps in the spectrum of the Maxwell operator or related scalar operators.
For more details we refer  to \cite{KK99,KK02}, where periodic high contrast photonic and acoustic media are studied in  high contrast regimes leading to appearance of   Dirichlet-to-Neumann type operators on periodic graphs.  

To create spectral gaps one can use geometrical means.
For example, given a fixed graph we ``decorate'' it
changing its geometrical structure at each vertex:
either one attaches to each vertex a copy of certain fixed compact
graph \cite{Ku05} (see also  \cite{SA00} where similar idea was used for discrete graphs) or
in each vertex one disconnects the edges
emerging from it and then connects their loose endpoints by a
certain additional graph (``spider'') \cite{O06,DKO16}.
 
Another way to open spectral gaps is to use  ``advanced'' vertex conditions.
For example, as we already noted the spectrum of the  Kirchhoff Laplacian
on a  rectangular lattice has no gaps, however (see \cite{Ex96}) if we replace
 Kirchhoff conditions by the so-called $\delta$-conditions 
 of the strength
$\al\not=0$ one immediately gets infinitely many gaps
provided
the lattice-spacing ratio is a rational number. 
 
Further results on spectral gaps opening for periodic quantum graphs as
well as on various estimates on their location and lengths   can be found in \cite{AEL96,BGP07,EG96,ET10,ET17,KS16,LP08,KP07,KL07,N15,N16,N14,N19,BKo10,Ko08}.\smallskip

When designing materials with prescribed properties it is desirable not only to open up spectral gaps, but also   
be able to control their location and length -- via a suitable choice of operator coefficients or/and geometry of the medium.
We addressed this problem for various classes of periodic operators in a series of papers \cite{Kh12,Kh13,Kh14,BK15,EK18}.
In particular, periodic quantum graphs were treated in \cite{BK15}. 
In this paper the required structure for the spectrum is achieved 
via the combination of   two approaches described above: taking a fixed  
periodic graph $\Gamma_0$ we decorate it attaching to each period cell $m$
  compact graphs $Y_{ij}$; here $j=1,\dots,m$, while the subscript $i\in\Z^n$ indicates to which period cell we attach $Y_{ij}$ (see Figure~\ref{fig0}, here $m=2$). 
  On $\Gamma$ we considered the Hamiltonian $\H\e$ 
 defined by the operation\,\, $-\eps^{-1}{\d ^2/ \d x^2}$ on the edges and the Kirchhoff 
conditions in 
all its vertices except the points of attachment of $Y_{ij}$ to $\Gamma_0$ -- in 
these points we pose (a kind of) $\delta'$-conditions\footnote{For the definition of  $\delta$ and $\delta'$-conditions in the graph context see, e.g., \cite{Ex96}.}. Note, that the vertex conditions we dealt with in \cite{BK15} ``generate'' only
Hamiltonians with   $\inf(\sigma(\H\e))=0$.  
It was proven that 
 $\sigma(\H\e)$ has at least $m$ gaps for small enough $\eps$, these gaps converge (as $\eps\to 0$) to
some intervals $(A_j,B_j)\subset[0,\infty)$  whose location and lengths can be nicely controlled by  a suitable choice
of coupling constants standing in those $\delta'$-conditions and a suitable choice the ``sizes'' of attached graphs $Y_{ij}$. 

\begin{figure}[h]
\begin{center}
\begin{picture}(290,70)
\includegraphics{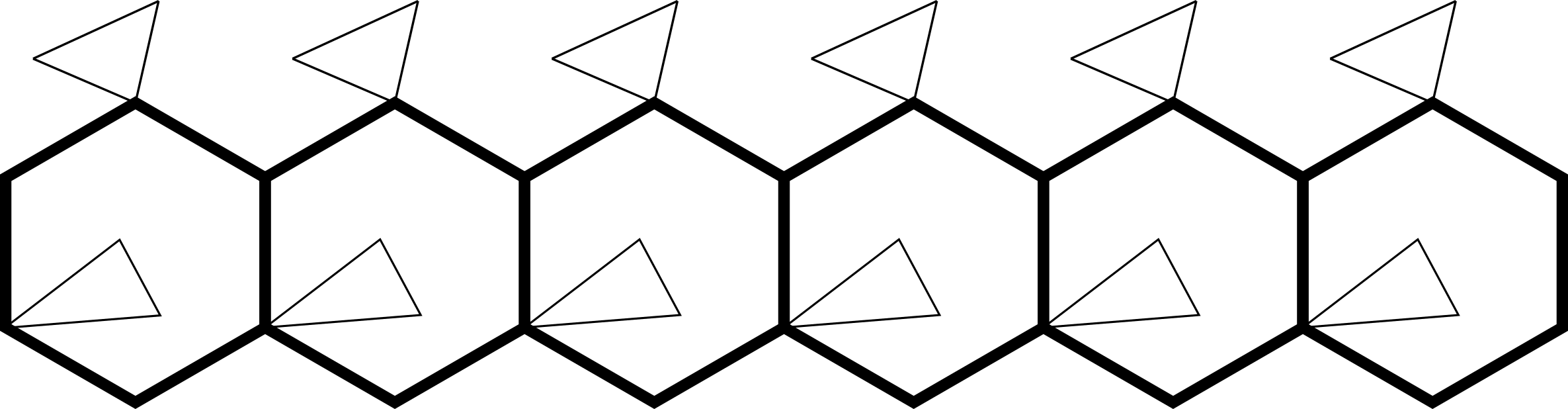}

\put(-3,50){$Y_{i1}$}
\put(-5,55){\vector(-3,1){18}}

\put(9,25){$Y_{i2}$}
\put(7,27){\vector(-3,-1){27.5}}

\put(-320,42){$\Gamma_0$}
\put(-308,41){\vector(3,-1){23}}

\end{picture}
\caption{Example of a periodic graph utilized in \cite{BK15}}
\label{fig0}
\end{center}
\end{figure}
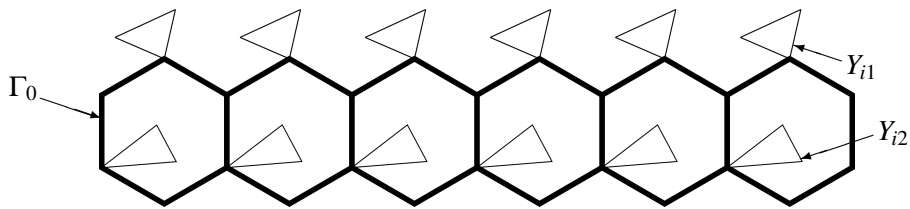
 
In the current paper we continue the research
started in \cite{BK15}. \emph{We will prove that the required structure of the spectrum
can be achieved solely by an appropriate choice of vertex conditions without any assumptions 
on the graph geometry}. Namely, let $\Gamma$ be a $\Z^n$-periodic metric graph. The only assumption we impose on it is that $\Gamma$ does not coincide with a line. On $\Gamma$ we consider  the Hamiltonian   $\H\e$ 
 defined by the operation $-\eps^{-1}{\d ^2/ \d x^2}$ on edges and either Kirchhoff, $\delta$ or $\delta'$-type (different from those treated in \cite{BK15}) conditions at  vertices -- see  \eqref{Kirchhoff}-\eqref{delta'}. 
We prove that  $\sigma(\H\e)$ has at least $m$ gaps; when $\eps\to 0$ the first $m$ gaps (respectively, the infimum  of $\sigma(\H\e)$)
converge to some intervals $(A_j,B_j)\subset\R$, $j=1,\dots,m$ (respectively, to some number $B_0\in\R$); the location of   $A_j$, $j=1,\dots,m$ and $B_j$, $j=0,\dots,m $ depends in  explicit way
from  couplings constants standing in $\delta$ and $\delta'$-type vertex conditions; see Theorem~\ref{th1}. Moreover, 
choosing these coupling constants in a proper way one can completely control  $A_j$ and $B_j$ 
making them  
coincident with predefined numbers; see Theorem~\ref{th2}. 
Note, that in contrast to \cite{BK15}, the limiting intervals and the bottom of the spectrum 
do not necessary lie on the positive semi-axis.  
Finally  we show that for fixed (small enough) $\eps$
one can  guarantee the \emph{precise coincidence} of the left endpoints of the first $m$ gaps with prescribed numbers; see Theorem~\ref{th3}. 
 
The method we use to prove the convergence of spectra is different from the one used in \cite{BK15},
where we utilized Simon's result  \cite{Si78} about monotonic sequences of forms.
In the current work we apply the abstract lemma
from \cite{EP05} serving to compare eigenvalues of two self-adjoint operators acting in different Hilbert spaces.  
The advantage of this approach is that we are able not only to prove the convergence of spectra, but also to estimate the rate of convergence.

The structure of the paper is as follows.
In Section~\ref{sec1} we introduce the Hamiltonian $\H\e$ and formulate the main convergence
result. Its proof  is given in Section~\ref{sec2}.
In Section~\ref{sec3} we demonstrate how to control the location of spectral gaps.

\section{Setting of the problem and main result\label{sec1}}
\subsection{Metric graph $\Gamma$\label{subsec11}} 

Let $n\in\N$ and let $\Gamma $ be an arbitrary connected $\Z^n$-periodic locally finite metric graph. 
The only assumption we impose on the geometry of $\Gamma$ is that it does not coincide with a line 
(see the footnote $^{\text{\ref{foot1}}}$   explaining the role of this assumptions) and 
its fundamental domain is compact (see below).
W.l.o.g. (cf.~the discussion after Definition~4.1.1 in \cite{BK13}) one can assume that $\Gamma$ is embedded into  $\mathbb{R}^d$ with  $d=n$ as $n\ge 3$ and $d=3$ as $n=1,2$. We also assume that $\Gamma$ has no loops  -- otherwise one can break them into pieces by introducing a new
intermediate  vertex.

By $\E_\Gamma$ and $\V_\Gamma$ we denote  the sets  of   edges and  vertices of $\Gamma$, respectively. By  
$l = l(e)$ we denote the
function assigning to the edge $e$ its length $l(e)$. 
We assume that $l(e) < \infty$ 
for each  $e \in \E_\Gamma$. In a natural way we introduce on  each edge $e\in\E_\Gamma$ 
the local coordinate $x_e \in [0, l(e)]$, so that  $x_e=0$ and $x_e=l(e)$ correspond to the endpoints of $e$. 
For $v \in \V_\Gamma$ we denote by
$ \E(v)$ the set of edges emanating from $v$.
 
 The $\mathbb{Z}^n$-periodicity of $\Gamma$ means  that
$$\Gamma+\nu_k=\Gamma,\ k=1,\dots,n$$
for some linearly 
independent vectors $\nu_1,\dots,\nu_n\in\R^d$.
Let us introduce for $i=(i_1,\dots,i_n)\in\Z^n$   the mapping $i\cdot:\Gamma\to\Gamma$  defined by
\begin{gather}\label{idot}
i\cdot x=  x+\suml_{k=1}^n i_k \nu_k,\ x\in\Gamma.
\end{gather}
We denote by $Y$  a \textit{fundamental domain} of $\Gamma$, i.e. a compact  set (see the assumption above) 
satisfying
$$
\ds\cupl_{i\in\Z^n} {i \cdot Y} = \Gamma,\quad
\text{the sets }Y\text{ and }i\cdot Y
\text{ may have only finitely many common points as }i\not=0.
$$
Evidently   a fundamental domain in not uniquely defined.
Note that for any $r=(r_1 , \dots , r_n ) \in \N^n_0$ the graph $\Gamma$
is also invariant invariant under translations through vectors $\nu^r_1,\dots,\nu^r_n$ defined by
$\nu_k^r=r_k\nu_k$
The corresponding fundamental domain is the set $Y^r$ given by
\begin{gather}
\label{rescalled}
Y^r= \cupl_{i\in \mathcal{I}^r} {i \cdot Y} ,\text{ where }
\mathcal{I}^r=\left\{i\in\Z^n:\ i_k\in[0,r_k],\, k=1,\dots,n\right\}.
\end{gather}
Finally,  we denote by $\mathcal{U}_Y $ the set of  points of a fundamental domain $ {Y}$ that 
simultaneously belong to ``neighboring'' fundamental domains, i.e.
$$\mathcal{U}_Y=\left\{v\in  {Y}:\ \exists i\in\Z^n\setminus\{0\}\text{ such that }v\in {i\cdot Y}\right\}.$$

An example of a $\Z^2$-periodic graph is presented on Figure~\ref{fig1}(a). This is
an equilateral  hexagonal lattice in $\R^2$, which is  invariant under translations through  
vectors $\vec \nu_1=(\sqrt{3},0)$, $\vec \nu_2=(-{\sqrt{3}\over 2},{3\over 2})$. Its fundamental domain $Y$   is highlighted in bold lines. On Figure~\ref{fig1}(b) one sees  the  fundamental domain $Y^r$ \eqref{rescalled} for $r=(2,2)$. On this figures the bold   dots  are vertices belonging to 
$\mathcal{U}_{Y }$ and $\mathcal{U}_{Y^r}$, respectively.  

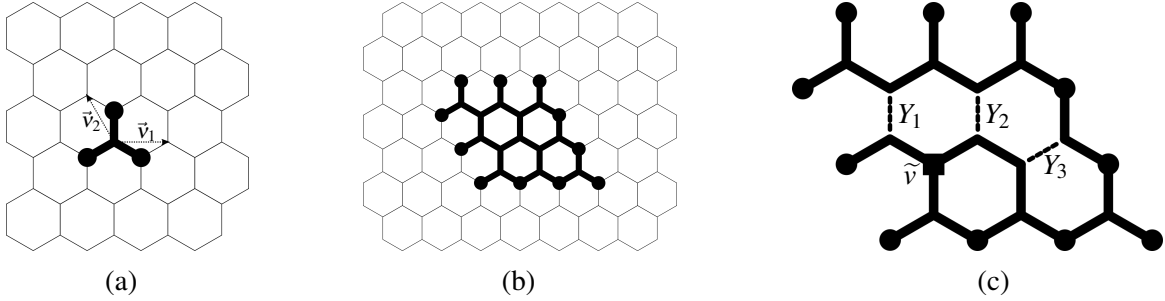
\begin{figure}[h]
\begin{center}
\begin{picture}(440,85)
\includegraphics[height=33mm]{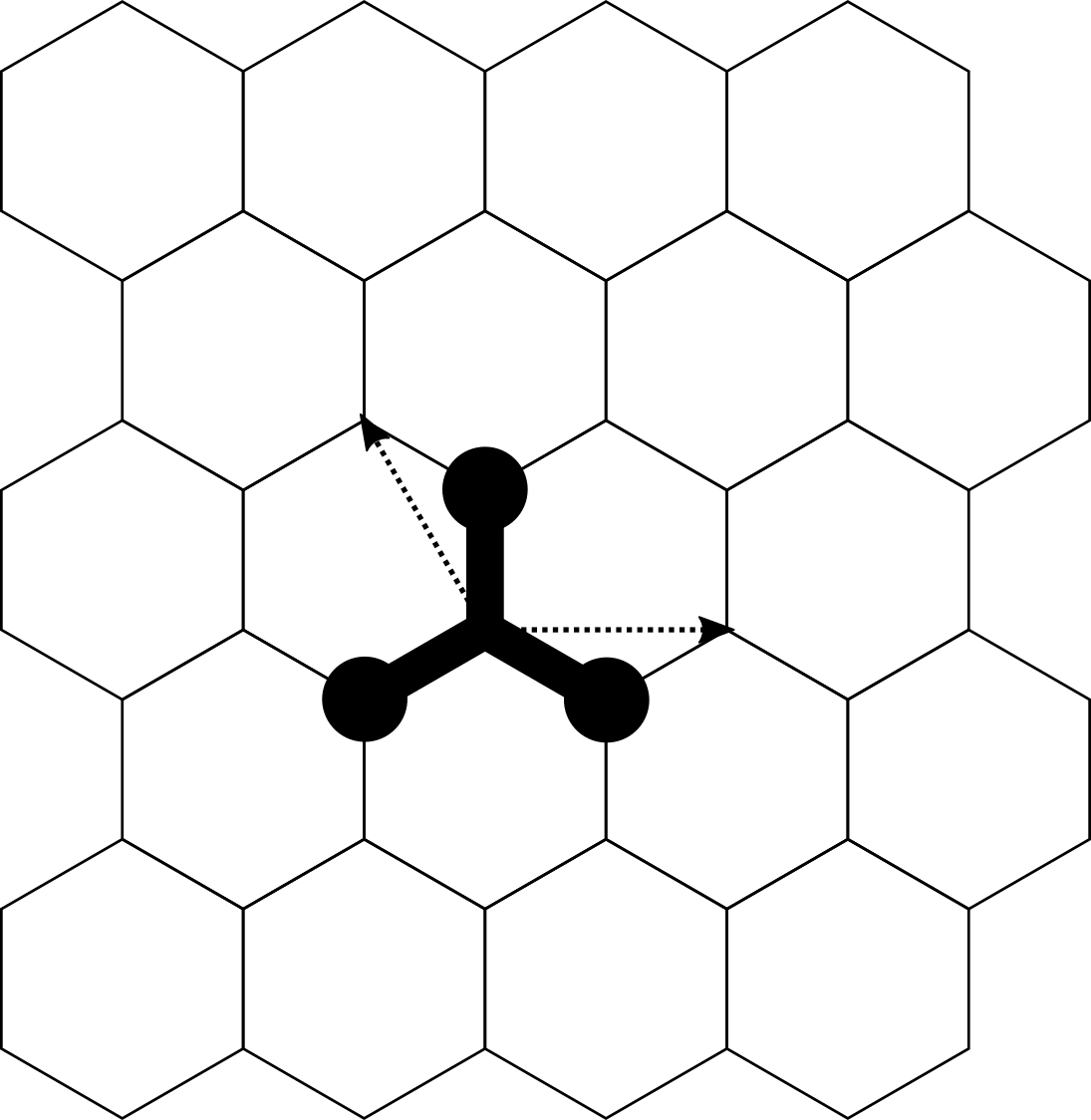}\qquad\qquad
\includegraphics[height=33mm]{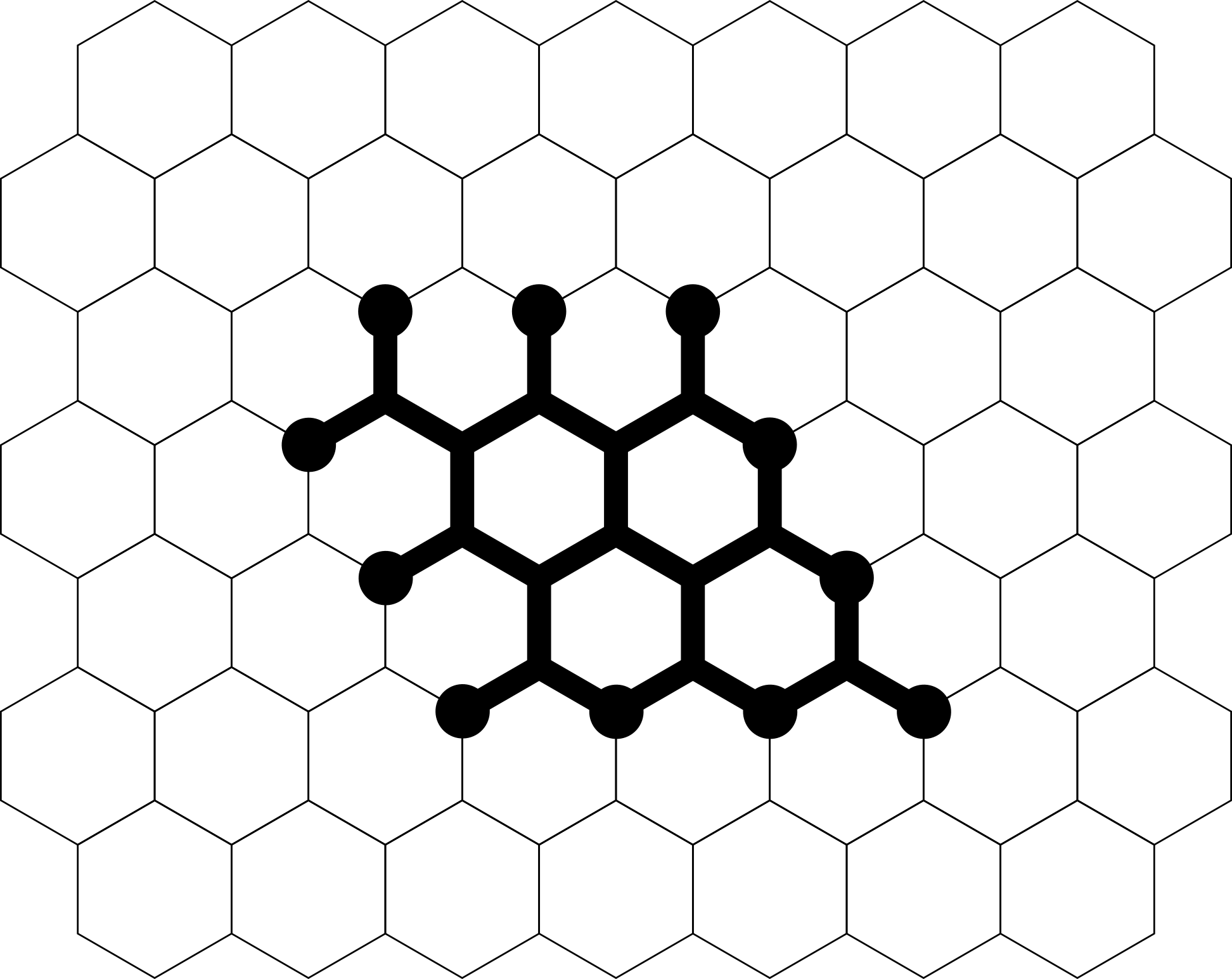}\qquad\qquad
\includegraphics[height=33mm]{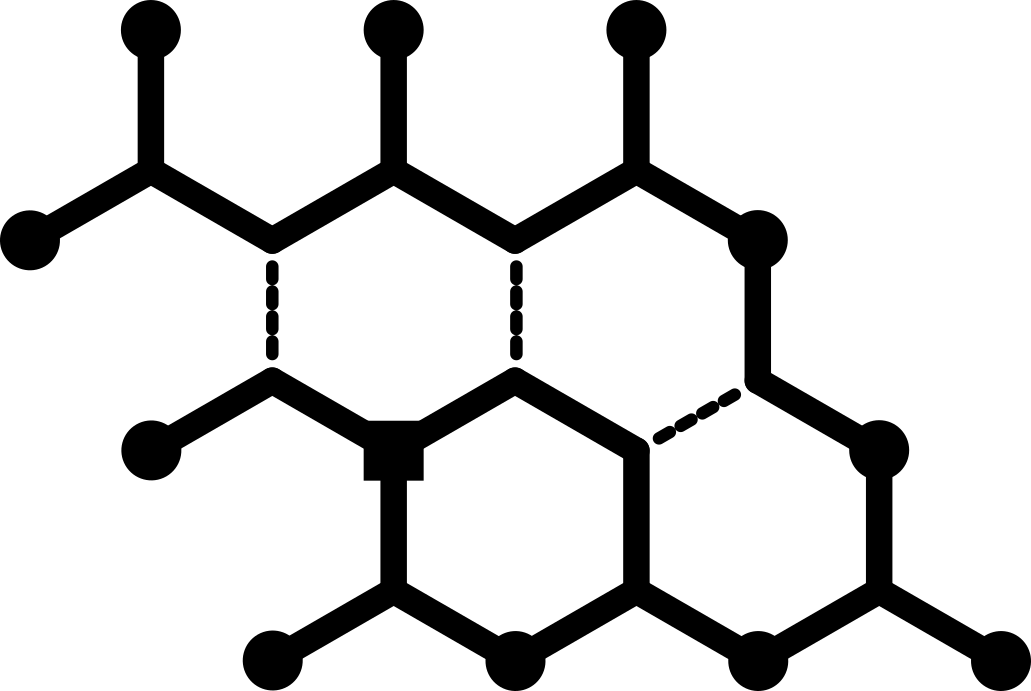}

\put(-100,48){$Y_1$}
\put( -67,48){$Y_2$}
\put( -45,29){$Y_3$}

\put( -98,25){$\widetilde v$}

\put( -409,49){$_{\vec \nu_2}$}
\put( -388,45){$_{\vec \nu_1}$}

\put(-400,-15){\text{(a)}}
\put(-250,-15){\text{(b)}}
\put(-70,-15){\text{(c)}}
\end{picture}\vspace{5mm}
\caption{
(a) $\Z^2$-periodic graph $\Gamma$ and its fundamental domain. \\
(b) The fundamental domain  $Y^r$ for $r=(2,2)$.\\  
(c) Decomposition of $Y^r$ for $m=3$.
\label{fig1}}
\end{center} 
\end{figure}

\subsection{Decomposition of a fundamental domain\label{subsec11+}} 
 
It is easy to see that for any $m\in \N$ there exists such  
$r=(r_1 , \dots , r_n ) \in \N^n_0$ that the fundamental domain $Y^r$ \eqref{rescalled}
can be represented as a union 
\begin{gather}\label{Ydecomp}
Y^r= \cupl_{j=0}^m  {Y_j} 
\end{gather}
of non-empty compact sets $Y_j $, $j=0,\dots,m$ satisfying the following conditions:
\begin{gather}
\label{assumptions}
\begin{array}{cl}
\rm(i)&Y_j\text{ are  connected},\ j=0,\dots,m,
\\ 
\rm(ii)&\mathcal{U}_Y \subset  {Y_0},\quad \mathcal{U}_Y \cap  {Y_j}=\varnothing,\ j=1,\dots m,
\\ 
\rm(iii)&\forall (j\not=0,\, k \not= 0,\, j \not= k):\ 
   {Y_j} \cap  {Y_k} = \varnothing,
\\ 
\rm(iv)& \text{the sets } \V_j:={Y_j} \cap    {Y_0} \text{ are non-empty  and consist of vertices}, \ j = 1,\dots , m,
\\ 
\rm(v)&
Y_0\text{ has a vertex }\wt v 
\text{ belonging neither to }\mathcal{U}_Y \text{ nor to }  \cup_{j=1}^m  {Y_j} .
\end{array}
\end{gather}
It is easy to see that such a decomposition is always possible for  
\textit{large enough} $r_1, r_2 , \dots , r_n$ 
\footnote{\label{foot1} In order to  achieve   decomposition \eqref{Ydecomp}-\eqref{assumptions} we require our initial assumption on $\Gamma$ that it does not coincide with a line. 
If $\Gamma$ is a line, its fundamental domain $Y^r$ would be a compact interval;
one can decompose it in such a way that properties (ii)-(v) hold, but then the set $Y_0$ will be always disconnected.}. Of course such a decomposition is not unique.
For example on  Figure~\ref{fig1}(c) the domain $Y^r$   is decomposed  in such a way that  \eqref{Ydecomp}-\eqref{assumptions} with $m=3$ holds:
   $Y_0$ consists of bold solid lines,  while $Y_1,Y_2,Y_3$ consist of one  dashed  edge,
the black square is   $\widetilde v$.

Now, let  $m\in\N$ be given and let us fix such
$r=(r_1,\dots,r_n)\in\N^n_0$ that the fundamental domain $Y^r$ admits  representation \eqref{Ydecomp}-\eqref{assumptions}.
We set for $i\in\Z^n$\,:
\begin{gather*}
V_{ij}:=i^r\cdot \V_j,\ j=1,\dots,m,\quad Y_{ij}:=i^r\cdot Y_j,\ j=0,\dots,m,\quad \wt v_i:=i^r\cdot\wt v,
\end{gather*}
where  the mapping $i^r\cdot:\Gamma\to\Gamma$ is defined by 
$
i^r\cdot x=  x+\suml_{k=1}^n i_k r_k \nu_k,\ x\in\Gamma.
$

The vertices belonging to $\V_{ij}$ will support $\delta'$-type conditions, the vertices  $\wt v_i$ will support $\delta$-conditions, in the remaining vertices the Kirchhoff conditions will be posed.

\subsection{Functional spaces\label{subsec12}} 

In what follows if $u : \Gamma \to \C$ and $e \in \E_\Gamma$ then by $u_e$ we denote the
restriction of $u$ onto the interior of $e$. Via a local coordinate $x_e$ we identify $u_e$ with a function on $(0, l(e))$.

The space $\L (\Gamma)$ consists of  functions $u : \Gamma \to \C$ such that 
$u_e \in \L (0, l(e))$ for each edge $e$ and  
$$
\|u\|^2_{\L(\Gamma)}:=\suml_{e\in\E_\Gamma}\|u_e\|^2_{\L(0,l(e))}<\infty.
$$
The space $\wt\W^k(\Gamma)$, $k\in\N$ consists of   functions $u : \Gamma \to \C$ such that 
$u_e$ belongs to the Sobolev space $\W^k (0, l(e))$ for each edge $e$ and
$$
\|u\|^2_{\wt\W^k(\Gamma)}:=\suml_{e\in\E_\Gamma}\|u_e\|^2_{\W^k(0,l(e))}<\infty.
$$
By $\W^1_\h (\Gamma)$ we denote a subspace of $\wt\W^1(\Gamma)$
consisting of such  function $u\in \wt\W^1(\Gamma)$ that
\begin{itemize}

\item   if $v \in \V_\Gamma\setminus\left( \cup_{i\in\Z^n}\cup_{j=1}^m \V_{ij}\right)$ then 
$u$ is continuous at $v$, i.e. the limiting value of $u(x)$ when
$x$ approaches $v$ along $e \in \E(v)$ is the same for each   $e\in \E(v)$. We denote this value by $u(v)$;

\item if $v \in \V_{ij}=Y_{ij}\cap Y_{i0}$ for some $i = (i_1 , \dots , i_n ) \in \Z^n$, 
$j \in \{1, \dots , m\}$ then
\begin{itemize}

\item the limiting value of $u(x)$ when $x$ approaches $v$ along $e \in \E(v)\cap Y_{i0} $ is the same for each $e \in \E(v)\cap Y_{i0}$. We denote this value by $u_0 (v)$.

\item the limiting value of $u(x)$ when $x$ approaches $v$ along $e \in \E(v)\cap Y_{ij} $ is the same for each
  $e\in \E(v)\cap Y_{ij} $. We denote this value by $u_j (v)$.

\end{itemize}

\end{itemize}

\subsection{Operator $\H\e$\label{subsec13}}

Let $\eps>0$ be a small parameter.
In $\L (\Gamma)$ we introduce the quadratic form $\h\e $,
\begin{gather}
\label{form}
\h\e\la u,u \ra=\eps^{-1}\suml_{e\in\E_\Gamma}\|u_e'\|^2_{\L(0,l(e))}+
\suml_{i\in\Z^n}\suml_{j=1}^m\suml_{v\in\V_{ij}}\al_j|u_0(v)-\be_j u_j(v)|^2+\suml_{i\in\Z^n}
\gamma|u(\wt v_i)|^2
\end{gather}
on the domain $\dom(\h\e ) = \W^1_\h (\Gamma)$.
Here $\al_j$, $\be_j$, $\ga$ are real constants, moreover $\al_j\not=0$, $\be_j\not=0$
(this assumption is needed to avoid the  decoupling at the vertex $v$,
cf.~\eqref{delta'}). These constants are on our disposal and they will be specified later in Section~\ref{sec3}.
The second and third terms in the right-hand-side of
\eqref{form} are indeed finite on $u \in \wt\W^1  (\Gamma)$, this follows easily  from the  trace inequality
\cite[Lemma 1.3.8]{BK13}
\begin{gather*}
|u(0)|^2 \leq 2l^{-1} \|u\|^2_{\L (0,l)} + l\|u'\|^2_{\L (0,l)},\ \forall u \in \W^1 (0, l)
\end{gather*}
and periodicity of $\Gamma$. It is also straightforward to verify that the form $\h\e$ is densely defined in $\L(\Gamma)$, lower
semibounded and closed.  By the
first representation theorem  \cite[Theorem~VI.2.1]{Ka66}  there exists the unique
self-adjoint  operator $\H\e$ associated with the form $\h\e $, i.e.
\begin{gather}\label{FRT}
(\H\e  u, w)_{\L  (\Gamma)} = \h\e  \la u,w \ra,\
\forall u \in \dom(\H\e  )\subset \dom(\h\e  ),\ \forall w \in \dom(\h\e  ),
\end{gather}
where $\h\e  \la u,w \ra$ is the  sesquilinear form, which corresponds  to the quadratic form \eqref{form}.

The domain of $\H\e$ consists of functions $u \in \W^1_\h (\Gamma) \cap \wt\W^2(\Gamma)$ satisfying
\begin{gather}
\label{Kirchhoff}
\ds\suml_{e\in\E(v)}\left.{\d u_e\over \d\x_e}\right|_{\x_e=0}=0\quad\text{at}\quad v\in\V_\Gamma\setminus
\cupl_{i\in \Z^n}\left(\{\wt v_i\}\cup\left(\cupl_{j=1}^m \V_{ij} \right)\right),
\\
\label{delta}
\ds\suml_{e\in\E(v)}\left.{\d u_e\over \d\x_e}\right|_{\x_e=0}=\gamma\eps\, u(v)\quad\text{at}\quad v=\wt v_i,\\
\label{delta'}
\left.
\begin{matrix}
\ds\suml_{e\in\E(v)\cap Y_{i0}}\left.{\d u_e\over \d\x_e}\right|_{\x_e=0}=\al_j\eps\left(u_0(v)-\be_ju_j(v)\right),\\
\ds\suml_{e\in\E(v)\cap Y_{ij}}\left.{\d u_e\over \d\x_e}\right|_{\x_e=0}=-\al_j\eps\left(u_0(v)-\be_ju_j(v)\right)
\end{matrix}
\right\}\quad\text{at}\quad v\in\V_{ij},
\end{gather}
where $\x_e\in[0,l(e)]$ is a natural coordinate on $e \in \E(v)$ such that $\x_e = 0$ at $v$. 
The action of $\H\e$ is 
\begin{gather}
\label{action}
(\H\e  u)_e = -\eps^{-1}{\d^2 u_e\over \d \x_e^2},\ e \in \E_\Gamma.
\end{gather}
Condition \eqref{Kirchhoff} is usually referred  as Kirchhoff coupling, condition \eqref{delta} is known
as  $\delta$-coupling of the strength $\gamma\eps$. We may refer to conditions \eqref{delta'} as 
$\delta'$-type coupling. The
reason for this is as follows. Suppose that $v \in \V_{i j}$ has only two outgoing edges 
$e \in \E(v) \cap Y_{i 0} $ and $\wt e \in \E(v) \cap Y_{ij}$. 
Also let $\be_j = 1$. Then  conditions \eqref{delta'} are equivalent to
$$ 
\left.{\d u_e\over \d\x_e}\right|_{\x_e=0}+\left.{\d u_{\wt e}\over 
\d\x_{\wt e}}\right|_{\x_{\wt e}=0}=0,\quad
(\al_j\eps)^{-1}\left.{\d u_e\over \d\x_e}\right|_{\x_e=0}=u_e|_{\x_e=0}-u_{\wt e}|_{\x_{\wt e}=0}.
$$
Taking into account the definition of   coordinates $\x_e$ and $\x_{\wt e}$  we conclude that \eqref{delta'} coincides
with the usual $\delta'$-conditions of the strength $(\al_j\eps)^{-1}$ 
at a point on the line \cite[Section~I.4]{AGHH05}.

\subsection{Main results\label{subsec14}} 

We denote
\begin{gather}\label{lj}
l_j:=\ds\suml_{e\in \E_{Y_j}}l(e),\ j=0,\dots,m,\qquad
N_j:=\textrm{\,cardinality of }\V_j,\ j=1,\dots,m.
\end{gather}
Then for $j = 1, \dots , m$ we set
\begin{gather}
\label{Aj}
A_j:={\al_j \beta_j^2 N_j l^{-1}_j}.
\end{gather}
We assume that $A_j$ are pairwise distinct; in this case we can renumber them in such a way that
\begin{gather}
\label{Amono}
\forall j = 1, \dots , m - 1:\quad
A_j < A_{j+1}.
\end{gather}
Finally, we consider the following equation (for unknown $\lambda \in \C\setminus\{A_1,\dots,A_m\}$)\,:
\begin{gather}
\label{equat}
\lambda\left(l_0+\suml_{j=1}^m {A_jl_j  \over  \be_j^2(A_j-\lambda)}\right)=\gamma.
\end{gather}
It is easy to show that this equation has exactly 
$m + 1$  roots $B_j$, $j = 0, \dots , m$, they are real, moreover (after an appropriate renumeration) these roots satisfy
\begin{gather}
\label{interlace}
B_{0 }<A_{1 }<B_{1 }<A_{2}<B_{2}<\dots<A_{m }<B_{m }.
\end{gather}
We are now in position to formulate the first main result of this work.

\begin{theorem}
\label{th1}
There exist  such positive constants $\Lambda_0$ (depending on $Y$) and $C_A,\,C_B,\,\eps_0$ (depending on $\al_j$, $\be_j$, $\ga$ and $Y$)   that the spectrum of $\mathcal{H}\e$ has the following structure within $(-\infty,\Lambda_0\eps^{-1}]$\,:
\begin{gather}\label{th1:spec}
\sigma(\H\e)\cap (-\infty,\Lambda_0\eps^{-1}]=[B_{0,\eps},\Lambda_0 \eps^{-1}]\setminus\cupl_{j=1}^m (A_{j,\eps},B_{j,\eps})\quad\text{as }\eps<\eps_0,
\end{gather} 
where the  numbers\, $A_{j,\eps}$, $j=1,\dots,m$\ \, and\ \, $B_{j,\eps}$, $j=0,\dots,m$\, satisfy
\begin{gather}\label{th1:inter}
B_{0,\eps}<A_{1,\eps}<B_{1,\eps}<A_{2,\eps}<B_{2,\eps}<\dots<A_{m,\eps}<B_{m,\eps}<\Lambda_0\eps^{-1}, 
\end{gather}
moreover
\begin{gather}
\label{th1:estim}
0\le A_j - A_{j,\eps}\leq C_A\eps^{1/2},\ j=1,\dots,m,\qquad
0\le B_j  - B_{j,\eps}\leq C_B\eps^{1/2},\ j=0,\dots,m.
\end{gather}
\end{theorem}

\section{Proof of Theorem~\ref{th1}\label{sec2}}

\subsection{Preliminaries\label{subsec31}}

To simplify the notations we assume that the fundamental domain $Y^r$ 
admits representation \eqref{Ydecomp}-\eqref{assumptions} for $r=0$, i.e., already the initial fundamental domain $Y$ admits such a representation. In the general case one should simply change the notations accordingly. 

In the following if $\H$ is a self-adjoint lower semi-bounded operator with purely discrete spectrum, we denote 
by $\left\{\lambda_k(\H)\right\}_{k\in\mathbb{N}}$ the sequence of its eigenvalues arranged in the ascending order and repeated according to their multiplicity.   
\smallskip

The Floquet-Bloch theory \cite[Chapter 4]{BK13} establishes a relationship between the spectrum of the operator 
$\H\e$ and the spectra of certain operators $\H\e^\theta$ in $\L(Y)$.
Namely, let $$\theta\in 
\mathbb{T}^n=\left\{\theta=(\theta_1,\dots,\theta_n)\in\mathbb{C}^n,\ 
|\theta_k|=1\text{ for all }k=1,\dots,n\right\}.$$ 
We denote by $\W^{1,\theta}_\h(\Gamma)$ the  set of such functions 
$u:\Gamma\to\C$  that $u_e\in \W^1(0,l(e))$ for each $e\in\E_\Gamma$,
$u$ satisfy the same conditions at vertices of $\Gamma$ as functions from $\W^{1}_\h(\Gamma)$,
and
$$\forall x\in\Gamma,\ \forall i=(i_1,\dots,i_n)\in\Z^n:\ 
u(i \cdot x)= \theta^i u(x),\text{ where }\theta^i:=\left(\prod\limits_{k=1}^n(\theta_k)^{i_k}\right)$$
(recall, that the mapping $i\cdot:\Gamma\to\Gamma$ is defined by \eqref{idot}).
We introduce the quadratic form $\h\e^\theta$  by
\begin{gather}\label{form:action}
\begin{array}{l}
\h^\theta\e\la u,u\ra=\ds\eps^{-1}\suml_{e\in\E_Y}\|u_e'\|^2_{\L(0,l(e))}
+\suml_{j=1}^m \suml_{v\in \V_{j}} 
\al_j\left|u_{0}(v)-\be_j u_{j}(v)\right|^2+\gamma \left|u(\wt v)\right|^2,\\
\mathrm{dom}(\h\e^\theta)=\left\{u=v|_Y,\ v\in\W_\h^\theta(\Gamma)\right\}.
\end{array}
\end{gather}
Hereinafter by $\E_Y$ and $\V_Y$  we denote the set of edges and vertices of $Y$, respectively;
similar notations will be used for $Y_j$.
The form $\h\e^\theta$ is  densely defined in $\L(Y)$, lower
semibounded and closed. We denote by $\H^\theta_{\eps}$   the operator  
associated with  $\h^\theta_{\eps}$.
The spectrum of $\H^\theta_{\eps}$ is purely discrete, moreover
for each $k\in\N$ the function $\theta\mapsto \lambda_k(\H^\theta_{\eps})$ is continuous. 
Consequently, the set
\begin{gather}\label{repres1+}
 L_{k,\eps}=\cupl_{\theta\in \mathbb{T}^n} 
\left\{\lambda_k(\H^\theta_{\eps})\right\}\text{ is a compact interval}.
\end{gather}
According to the Floquet-Bloch theory we have the following representation:
\begin{gather}\label{repres1}
\sigma(\H\e)=\cupl_{k=1}^\infty L_{k,\eps}.
\end{gather}

Along with  $\h\e^\theta$ we also introduce the forms $\h\e^N$ and $\h\e^D$ 
acting on the domains
$$
\dom(\h\e^N)=\left\{u=v|_Y,\ v\in\W_\h(\Gamma)\right\}\text{\quad and\quad }
\dom(\h\e^D)=\left\{u=v|_Y,\ v\in\W_\h(\Gamma),\ \mathrm{supp}(v)\subset {Y}\right\}
$$
and  with the action being again specified by \eqref{form:action}. By 
$\H\e^N$ and $\H\e^D$ we denote the associated operators. The spectra of these operators  are purely discrete.  It is easy to see  that
$$
\forall\theta\in \mathbb{T}^n:\quad 
\dom(\h\e^N)\supset\dom(\h\e^\theta)\supset\dom(\h\e^D),
$$
whence, using the min-max principle \cite[Section~4.5]{De95}, we conclude
\begin{gather}\label{enclosure}
\forall k\in \mathbb{N},\ \forall\theta\in \mathbb{T}^n:\quad
\lambda_{k}(\H\e^N) \leq \lambda_{k}(\H\e^\theta) \leq
\lambda_{k}(\H\e^D).
\end{gather}

In the following we  mostly use two distinguished points of $\mathbb{T}^n$,
\begin{gather}\label{pa}
\theta_{p}:=(1,1,\dots,1)\quad{and}\quad \theta_{a}:=-(1,1,\dots,1).
\end{gather}
The subscripts $p$ and $a$ means \emph{periodic} and \emph{antiperiodic}, respectively.

\begin{remark}
Despite in the following we  work only on  the level of  forms  one is curious to take a more close
look  on the  operators $\H^\theta\e$, $\H^N\e$, $\H^D\e$.
Let $u\in\dom(\H^*\e)$ with $*\in\{\theta,N,D\}$. Then
\begin{itemize}
\item for each $e\in\E_Y$ one has $u_e\in\W^2(0,l(e))$, 

\item at the vertices from $ \V_Y\setminus \mathcal{U}_Y$ $u$ satisfies the same conditions  as functions belonging to $\dom(\H\e)$.

\end{itemize}
To describe the behaviour of $u$ on $\mathcal{U}_Y$ we assume for simplicity that points of $\mathcal{U}_Y$ lye on the \emph{interiors} of edges  of $\Gamma$ (in fact, one can always choose 
a period cell in such a way that this assumption fulfills).  This assumption on a period cell implies, in particular, that for any $v\in\mathcal{U}_Y$ there is only one edges of $ \E_Y$ (we denote it $e_v$) emanating from $v$. 
Then we get the following boundary conditions at $\mathcal{U}_Y$:
\begin{itemize}
\item $u\in\dom(\H^\theta\e)$ satisfies $\theta$-periodic conditions at $v\in\mathcal{U}_Y$:
$$
u_{e_{w}}(w)=\theta^i u_{e_v}(v),\quad 
{\d u_{e_{w}}\over \d\x_{e_{  w}} }(w)=-\theta^i {\d u_{e_{ v}}\over \d\x_{e_{v}} }(v)
$$ 
where $w\in \mathcal{U}_Y$ is such  
that $w=i\cdot v$ for some $i\in\Z^n$ (one can show that for each $v\in \mathcal{U}_Y$ there exists a unique $w$ with $w=i\cdot v$ for some $i\in\Z^n$ provided the period cell is chosen as above),
\item $u\in\dom(\H^N\e)$ satisfies Neumann conditions ${\d u_e\over \d\x_e }(v)=0$ at $v\in\mathcal{U}_Y$,
\item $u\in\dom(\H^D\e)$ satisfies Dirichlet conditions $u_e(v)=0$ at $v\in\mathcal{U}_Y$.
\end{itemize}
Above $\x_{e_v}\in [0,l(e_v)]$  is a natural coordinate on $e_v$ such that 
$\x_{e_v}=0$ at $v$; in the same way   $\x_{e_w}$ is defined.
The action of  all operators above  is given by \eqref{action}.  

\end{remark}

\subsection{Determination of $\Lambda_0$}

Recall, that $\theta_a\in\mathbb{T}^n$ is given in \eqref{pa}.

\begin{lemma}\label{lm1}
There exist
$\Lambda_0>0$ and $\eps_\Lambda>0$ such that
\begin{gather}\label{lm1:ineq}
\Lambda_0\eps^{-1}< 
\lambda_{m+1}(\H\e^{\theta_a})\quad\text{as }\eps<\eps_\Lambda.
\end{gather}
\end{lemma}

\begin{proof}
For  $\theta\in\mathbb{T}^n$ and $\eps\ge 0$ we introduce in $\L(Y)$ the   form  $\hb^\theta\e$,
\begin{gather*}
\ds\hb^\theta\e\la u,u\ra= \suml_{e\in\E_Y}\|u_e'\|^2_{\L(0,l(e))}+
\eps\suml_{j=1}^m \suml_{v\in \V_{j}} \al_j 
\left| u_{0}(v)-\beta_j u_j(v)\right|^2+
 \eps\gamma\left|u(\wt v)\right|^2,\quad \mathrm{dom}( \hb^\theta\e)=\mathrm{dom}(\h\e^\theta).
\end{gather*}
We denote by $ \Hb^\theta\e$ the self-adjoint operator associated 
with this form. Obviously,
\begin{gather}
\label{lambda:lambda}
\forall \eps>0\ \forall k\in\N:\quad \eps^{-1}\lambda_k(\Hb^\theta\e)= \lambda_k(\H^\theta\e).
\end{gather}
Also we observe that with respect to the space decomposition  $\L(Y)=\oplus_{j=0}^m\L(Y_j)$ the operator  $\Hb^\theta_0$ can be decomposed in a sum
\begin{gather}
\label{oplus}
\Hb^{\theta}_0=\Hb_{0,0}^\theta\oplus\Hb_{0,1}^N\oplus\Hb_{0,2}^N\oplus\dots\oplus\Hb_{0,m}^N, 
\end{gather}
where the operators $\Hb_{0,0}^\theta$, $\Hb_{0,j}^N$ are  associated 
with the forms $\hb_{0,0}^\theta$, $\hb_{0,j}^N$ defined as follows,
\begin{gather}\label{hb0}
\ds\hb_{0,0}^\theta\la u,u\ra= 
\suml_{e\in\E_{Y_0}}\|u_e'\|^2_{\L(0,l(e))},\quad
 \dom(\hb_{0,0}^\theta)=\left\{u=v|_{Y_0},\ v\in \dom(\h \e^\theta)\right\},
\\
\label{hbj}
\ds\hb_{0,j}^N\la u,u\ra= 
 \suml_{e\in\E_{Y_j}}\|u_e'\|^2_{\L(0,l(e))},\quad 
 \dom(\hb_{0,j}^N)=\left\{u=v|_{Y_j},\ v\in \dom(\h\e^\theta)\right\}.
\end{gather}
It is easy to see that 
\begin{gather}
\label{lambda:j:1}
\lambda_1(\Hb_{0,j}^N)=0,\ j=1,\dots,m,
\end{gather}
and the corresponding eigenspace consists of constant functions.
Due to the connectivity of $Y_j$ one has
\begin{gather}
\label{lambda:j:2}
\lambda_2(\Hb_{0,j}^N)>0,\ j=1,\dots,m.
\end{gather}
If   $\lambda_1(\Hb_{0,0}^\theta)=0$,  the corresponding eigenfunction would be constant 
which is possible iff $\theta=\theta_p$. Thus
\begin{gather}\label{lambda:0:1}
\lambda_1(\Hb_{0,0}^{\theta})>0,\ \theta\not=\theta_p.
\end{gather}
It follows from \eqref{oplus}, \eqref{lambda:j:1}-\eqref{lambda:0:1}
that $\lambda_k(\Hb_{0}^{\theta})=0$ for $k=1,\dots, m$, while
\begin{gather}\label{lambda:0:all}
\lambda_{m+1}(\Hb_{0}^{\theta})=\min\left\{\lambda_1(\Hb_{0,0}^{\theta});
\lambda_2(\Hb_{0,1}^N);\,\lambda_2(\Hb_{0,2}^N);\,\dots;\,\lambda_2(\Hb_{0,m}^N)\right\}>0,\ \theta\not=\theta_p.
\end{gather}
Using the fact that sequence of forms $\hb\e^\theta$ increases monotonically 
as $\eps$ decreases, and moreover 
$
\lim_{\eps\to 0}\hb\e^\theta\la u,u\ra =\hb_0^\theta\la u,u\ra$,
$
\forall u\in \mathrm{dom}(\hb\e^\theta)=\mathrm{dom}(\hb_0^\theta)
$
we conclude \cite[Theorem 4.1]{Si78}:
\begin{gather}
\label{strong}
\forall f\in\L(Y):\ \left\|(\Hb\e^\theta+\mathrm{I})^{-1}f-(\Hb_0^\theta+\mathrm{I})^{-1}f\right\|_{\L(Y)}\to 0\text{ as }\eps\to 0.
\end{gather}
Moreover, 
since the sequence of resolvents $(\Hb\e^\theta+\mathrm{I})^{-1}$ decrease monotonically as $\eps$ decreases, and both resolvents $(\Hb\e^\theta+\mathrm{I})^{-1}$ and $(\Hb_0^\theta+\mathrm{I})^{-1}$ are compact one can upgrade \eqref{strong} to
the norm resolvent convergence \cite[Theorem VIII-3.5]{Ka66}. As a consequence we get the convergence of spectra, namely
\begin{gather}\label{lambda:j:conv}
\forall k\in\N:\ \lambda_k(\Hb_{\eps}^\theta)\to \lambda_k(\Hb^\theta_{0})\text{ as }\eps\to 0.
\end{gather}
We set
\begin{gather}\label{Lambda0}
\Lambda_0:=\ds{\lambda_{m+1}(\Hb_{0}^{\theta_a})\over 2}.
\end{gather}
Since $\theta_a\not=\theta_p$, then $\Lambda_0>0$.
It follows from  \eqref{lambda:j:conv} that there exists 
$\eps_\Lambda>0$  such that
\begin{gather}  \label{lambda:L0}
 \Lambda_0<\lambda_{m+1}(\Hb\e^{\theta_a})  \text{ as }\eps<\eps_\Lambda.
\end{gather}
Combining \eqref{lambda:lambda} and \eqref{lambda:L0}
we arrive at the desired estimate \eqref{lm1:ineq}.
The lemma is proven.
\end{proof}

\subsection{Comparison of eigenvalues}
Here we recall a result from \cite{EP05} serving to compare eigenvalues of two operators acting in different Hilbert spaces.
Let $H$ and $H'$ be separable Hilbert spaces, $\H$ and $\H'$ be non-negative self-adjoint operators in these spaces, and $\h$ and $\h'$ be the associated quadratic forms. We assume that both operator $\H$ and $\H'$ have purely discrete spectra. 
 
\begin{lemma}\label{lm:EP} {\cite[Lemma~2.1]{EP05}}
Suppose that $\Phi:\dom(\h)\to\dom(\h')$ is a linear map such that  
\begin{gather*}
\|u\|^2_{H}\leq \|\Phi u\|_{H'}^2 + \delta_1\left(\|u\|_{H}^2+\h\la u, u\ra\right),\qquad
\h' \la \Phi u, \Phi u\ra\leq \h\la u, u\ra + \delta_2\left(\|u\|_{H}^2+\h\la u, u\ra\right)
\end{gather*}
for all $u\in \dom (\h)$.
Here  $\delta_1,\,\delta_2 $ are some positive constants. Then for each $j\in\N$ we have
\begin{gather}\label{est:EP}
\lambda_j(\H')\leq \lambda_j(\H) + {\lambda_j(\H)(1+\lambda_j(\H) )\delta_1+(1+\lambda_j(\H) )\delta_2\over 1-(1+\lambda_j(\H) )\delta_1}
\end{gather}
 provided the denominator $1-(1+\lambda_k(\H))\delta_1$ is positive.
\end{lemma}

\begin{remark}\label{remA}
 The above result was established in \cite{EP05} under the assumption that $\dim H=\dim H'=\infty$, however, it is easy to see from its proof that the result remains valid for $\dim H'<\infty$ as well. In that case \eqref{est:EP} holds for $j\in\{1,\dots,\,\dim H'\}$.  
\end{remark}

\subsection{Estimates on $\lambda_{k} (\H^N\e)$ and $\lambda_{k}(\H^{\theta_p}\e)$}
 
In this subsection  we  denote by bold letters (e.g., $\u$) the elements of $\C^{m+1}$. Their entries will be enumerated starting from zero, i.e.
$$
\u\in \C^{m+1}\ \Rightarrow\ \u=(u_0,\dots,u_m)\text{ with }u_j\in\C.
$$
Let $\C^{m+1}_l$ be the same space $\C^{m+1}$ equipped with the weighted scalar product
\begin{gather}
\label{scalar-product}
(\u,\v)_{\C^{m+1}_l}=\suml_{j=0}^m u_j \overline{v_j} l_j
\end{gather}
(recall that $l_j$ and $N_j$ are defined by \eqref{lj}).
Note that $\C^{m+1}_l$ is isomorphic to  a subspace of
$\L(Y)$ consisting of functions being constant on each $Y_j$, $j=0,\dots,m$.
In $\C^{m+1}_l$  we introduce the  form
\begin{gather}\label{hN0}
\h^N_0\la \u,\u\ra=\suml_{j=1}^m  
\al_j N_j \left|u_0-\be_j u_j\right|^2+\gamma \left|u_0\right|^2.
\end{gather}
This form is associated with the operator $\H^N_0$ in $\C^{m+1}_l$ being given by the symmetric (with respect to the scalar product \eqref{scalar-product}) matrix
\begin{gather*}
\H^N_0=\left(\begin{matrix}\gamma l_0^{-1}+\ds\suml_{j=1}^m \al_j N_j l_0^{-1}&-\al_1\be_1 N_1 l_0^{-1}&-\al_2\be_2 N_2 l_0^{-1}&\dots&-\al_m\be_m N_m l_0^{-1}\\-\al_1\be_1 N_1 l_1^{-1}&\al_1\be^2_1 N_1 l_1^{-1}&0&\dots&0\\[1mm]
-\al_2\be_2 N_2 l_2^{-1}&0&\al_2\be_2^2 N_2 l_2^{-1}&\dots&0\\
\vdots&\vdots&\vdots&\ddots&\vdots\\
-\al_m\be_m N_m l_m^{-1}&0&0&\dots&\al_m\be_m^2 N_m l_m^{-1}
\end{matrix}\right).
\end{gather*} 
We denote by $\lambda_1(\H^N_0)\le \lambda_2(\H^N_0)\le\dots\le \lambda_{m+1}(\H^N_0)$ its eigenvalues. 
It turns out that
\begin{gather}\label{lambdaB}
 \lambda_j(\H_0^N)= B_{j-1},\ j=1,\dots,m+1.
\end{gather}
Indeed, let $\lambda$ be the eigenvalue of $\H^N_0$ such that $\lambda\notin\{A_1,A_2,\dots,A_m\}$, and let
 $0\not=\u=(u_0,\dots,u_m)$ be the corresponding eigenfunction. The equation $\H^N_0\u=\lambda \u$
is a linear algebraic system for $u_0,\dots,u_m$. From the last $m$ equations of this system we infer
\begin{gather}\label{uu}
u_j={\al_j\be_j N_j l_j^{-1}\over \al_j\be_j^2 N_j l_j^{-1}-\lambda}u_0,\ j=1,\dots,m.
\end{gather}
Note, that the denominator in \eqref{uu} is non-zero since $\lambda\not= A_j=\al_j\be_j^2 N_j l_j^{-1}$.
Inserting \eqref{uu} into the first equation of the system we arrive at
$$u_0\left[\lambda\left(l_0+\suml_{j=1}^m {A_jl_j  \over \be_j^2(A_j-\lambda)}\right)-\gamma \right]=0.$$
Moreover, $u_0\not=0$ (otherwise, due to \eqref{uu}, $\u$ would vanish). Hence $\lambda$ is a root of equation \eqref{equat}. Evidently, the converse assertion also holds, that is
\begin{gather}\label{ev:root}
\lambda\in \sigma(\H^N_0)\setminus \{A_1,A_2,\dots,A_m\}\ \Longleftrightarrow\ \lambda\text{ is a root of }\eqref{equat}.
\end{gather}
Then \eqref{lambdaB} follows immediately from \eqref{interlace} and \eqref{ev:root}.

\begin{lemma}\label{lmN}
There exist such constants   $C_B>0$ and $\eps_B>0$   that
\begin{gather}
\label{lmN:2}
B_{j-1}\leq \lambda_j(\H\e^N)+C_B\eps^{1/2},\ j=1,\dots,m+1\quad\text{as }\eps<\eps_B.
\end{gather}
\end{lemma}	

\begin{proof}
W.l.o.g. we may assume that $\al_j$ and $\gamma$ are non-negative.
Evidently, under this assumption the operators $\H\e^N$ are non-negative. Moreover, the operator $\H_0^N$ is also non-negative, since
\begin{gather}\label{hh}
\h\e^N\la \Psi\u,\Psi\u\ra = \h_0^N \la \u,\u\ra,\ \forall \u\in\C^{m+1}_l, 
\end{gather}
where $\Psi:\C_l^{m+1}\to\L(Y)$ is defined by
\begin{gather}
\label{Psi}
\Psi\u=\suml_{j=0}^{m}u_j\chi_{Y_j},\quad\chi_{Y_j}\text{ is the indicator function of }Y_j
\end{gather} 
(it is easy to see that the image of $\Psi$ is contained in $\dom(\h\e^N)$).
Thus we are in the framework of Lemma~\ref{lm:EP}.
In the general case we have to consider the shifted operators $\H\e^N-\mu\mathrm{I}$ and $\H_0^N-\mu\mathrm{I}$, where $\mu$ is the smallest eigenvalues of $\H\e^N|_{\eps=1}$. The  operator $\H\e^N-\mu\mathrm{I}$ is non-negative for each $\eps\in(0,1]$ due to the fact that the sequence of forms $\h\e^N$ increases monotonically as $\eps$ decreases; the non-negativity of $\H_0^N-\mu\mathrm{I}$ follows from \eqref{hh}.  

We introduce  the operator $\Phi:\dom(\h^N\e)\to \C_l^{m+1}$ by
\begin{gather}\label{Phi}	
(\Phi u)_j={l^{-1}_j}\suml_{e\in\E_{Y_j}}\int_{0}^{l(e)}u_e(x)\d x,\quad j=0,\dots, m.
\end{gather}	
Our goal is to show that the following estimates hold for each $u\in\dom(\h\e^N)$:
\begin{gather}
\label{estim1}
\|u\|_{\L(Y)}^2\leq \|\Phi u\|_{\C_l^{m+1}}^2 + C_1\eps\left(\|u\|^2_{\L(Y)}+\h\e^N\la  u, u\ra\right),\\
\h_0^N\la \Phi u,\Phi u\ra\leq \h^N\e\la u,u\ra + C_2\eps^{1/2}\left(\|u\|^2_{\L(Y)}+\h\e^N\la u,u\ra\right).\label{estim2}
\end{gather}
with some $C_1,C_2>0$.
By  Lemma~\ref{lm:EP}  (see also Remark~\ref{remA} after it) we  infer from \eqref{estim1}-\eqref{estim2} that
\begin{gather}\label{EP:estim}
B_j=\lambda_{j}(\H_0^N)\leq \lambda_{j}(\H\e^N) + 
{\lambda_{j}(\H\e^N)(1+\lambda_{j}(\H\e^N))C_1\eps+(1+\lambda_{j}(\H\e^N))C_2\eps^{1/2} \over 1-(1+\lambda_{j}(\H\e^N))C_1\eps}
\end{gather}
provided $ (1+\lambda_{j}(\H\e^N))C_1\eps< 1$.
Taking into account that $0\le \lambda_{j}(\H\e^N)$ and $\lambda_{j}(\H\e^N)\le  B_{j-1}$ 
(this estimate follows from \eqref{enclosure} and Lemma~\ref{lmP} below), 
 we conclude  from \eqref{EP:estim} that there exists such $\eps_B>0$  that the required estimate \eqref{lmN:2}
holds for $\eps<\eps_B $.

To prove \eqref{estim1} we need a Poincar\'e-type inequality on each $Y_j$.
Namely, let the form $\hb^N_{0,j}$ be defined by \eqref{hbj}, $j=1,\dots,m$; in the same way we define
$\hb^N_{0,j}$ for $j=0$. By $\Hb^N_{0,j}$ we denote the associated   operators in $\L(Y_j)$, $j=0,\dots,m$.
One has $\lambda_1(\Hb^N_{0,j})=0$ (the corresponding eigenspace consists of constants),
while $\lambda_2(\Hb^N_{0,j})>0$.
By the max-min principle \cite{RS78}
$\lambda_2(\Hb^N_{0,j})\le  {\hb^N_{0,j}\la v,v\ra/\|v\|^2_{\L(Y)}}$ for each $v\in\dom(\hb^N_{0,j})$ such that
$(v,1)_{\L(Y_j)}=0$. Using the above estimate for $v:=u-(\Phi u)_j$ we get
\begin{gather}
\label{poincare}
\forall u\in \hb^N_{0,j}:\quad
\|u-(\Phi u)_j\|^2_{\L(Y_j)}\leq
C_1\suml_{e\in \E_{Y_j}}\|u_e'\|^2_{\L(0,l(e))},\ j=0,\dots,m,
\end{gather}
where $C_1=(\lambda_2(\Hb_{0,j}))^{-1}$.
Using \eqref{poincare} we obtain
\begin{multline*}
\|u\|_{\L(Y)}^2=\suml_{j=0}^m\|u\|_{\L(Y_j)}^2=
 \suml_{j=0}^m\left(|(\Phi u)_j|^2l_j+\|u-(\Phi u)_j\|_{\L(Y_j)}^2\right)\\ \leq
  \|\Phi u\|^2_{\C_l^{m+1}}+ C_1\suml_{e\in \E_{Y}}\|u_e'\|^2_{\L(0,l(e))}\leq
  \|\Phi u\|^2_{\C_l^{m+1}}+ C_1\eps\left(\|u\|^2_{\L(Y)}+ \h\e^N\la u,u\ra\right)
\end{multline*}
(on the last step we use the fact that $\al_j$ and $\gamma$ are non-negative).
Inequality \eqref{estim1} is checked. 

Now let us prove the estimate \eqref{estim2}. One has:
\begin{multline}\label{R1}
\h^N_0\la \Phi u,\Phi u\ra=
\h^N\e\la u,u\ra-\suml_{e\in\E_Y}\|u_e'\|_{\L(0,l(e))}^2\\
+
\underset{=:R\e }{\underbrace{\suml_{j=1}^m\suml_{v\in \V_{j}} \al_j\left[\left|(\Phi u)_0-\be_j (\Phi u)_j\right|^2-
\left|u_{0}(v)-\be_j u_{j}(v)\right|^2\right]
+\gamma\left[ \left|(\Phi u)_0\right|^2
-\left|u(\wt v)\right|^2\right]}}\\ \leq \h^N\e\la u,u\ra+R\e .
\end{multline}
We estimate the remainder $R\e$ as follows (below we use the estimate $|a|^2-|b|^2\leq |a-b|\left(|a|+|b|\right)$):
\begin{multline}\label{R2}
|R\e|\leq  \suml_{j=1}^m\suml_{v\in \V_{j}}\al_j
\bigg(|(\Phi u)_0-u_0(v)|+|\beta_j|\cdot|(\Phi u)_j-u_j(v)|\bigg)\\\times
\bigg(|(\Phi u)_0|+|u_0(v)|+|\beta_j|\cdot|(\Phi u)_j|+|\beta_j|\cdot|u_j(v)|\bigg) +
\gamma\,|(\Phi u)_0-u(\wt v)|\cdot \big(|(\Phi u)_0|+|u(\wt v)|\big).
\end{multline}
To proceed further we need a standard trace estimate 
\begin{gather}
\label{trace}
\|w\|_{\mathsf{L}^\infty(Y_j)}\leq \wt C\|w\|_{\W^1(Y_j)},\ w\in \W^1(Y_j),\ j=0,\dots,m,
\end{gather}
where $\wt C>0$ depends on $Y$.
Applying it for $w:=u\restriction_{Y_j}-(\Phi u)_j$ and then using   \eqref{poincare} we obtain  
\begin{multline}\label{uest1}
 \|u-(\Phi u)_j\|_{\mathsf{L}^\infty(Y_j)}
\leq \wt C\left( {\|u-(\Phi u)_j\|_{\L(Y_j)}+
\suml_{e\in \E_{Y_j}}\|u_e'\|^2_{\L(0,l(e))}}\right)^{1/2}\leq
\wt C(C_1+1) \left({\suml_{e\in \E_{Y_j}}\|u_e'\|^2_{\L(0,l(e))}}\right)^{1/2}\\\leq
\wt C(C_1+1) \eps^{1/2}\,\left(\|u\|_{\L(Y)}^2+\h^N\e\la u,u\ra\right)^{1/2},\ j=0,\dots,m.
\end{multline}
Also, using the Cauchy-Schwarz inequality and \eqref{trace} and taking into account that $\eps\le 1$, one gets
\begin{gather}\label{uest2}
|(\Phi u)_j|\leq l_j^{-1/2}\|u\|_{\L(Y_j)}\le \max_j l_j^{-1/2}\left(\|u\|_{\L(Y)}^2+\h\e^N\la u,u\ra\right)^{1/2},\ j=0,\dots,m,\\
\label{uest3}
 \|u\|_{\mathsf{L}^\infty(Y_j)} \leq \wt C \left(\|u\|_{\L(Y)}^2+\h\e^N\la u,u\ra\right)^{1/2},\ j=0,\dots,m.
\end{gather}
Combining \eqref{R2}, \eqref{uest1}-\eqref{uest3} we arrive at the estimate
\begin{gather}\label{R3}
|R\e|\leq C_2\eps^{1/2}\left(\|u\|_{\L(Y)}^2+\h\e^N\la u,u\ra\right)
\end{gather}
with some constant $C_2$ depending on $\al_j,\,\be_j,\,\gamma,\,Y$.
The required estimate \eqref{estim2} follows from \eqref{R1}, \eqref{R3}; this ends the proof of Lemma~\ref{lmN}.
\end{proof}

\begin{lemma}\label{lmP}
One has:
\begin{gather*}
\lambda_j(\H\e^{\theta_p})\le B_{j-1},\ j=1,\dots,m+1.
\end{gather*}
\end{lemma}	  

\begin{proof}
By the min-max principle \cite[Section~4.5]{De95} we have
\begin{gather}\label{minmax}
\lambda_j(\H\e^{\theta_p})=\min_{V\in \mathfrak{H}^j}\max_{u\in V\setminus\{0\}}{\h\e^{\theta_p}\la u,u\ra	\over \|u\|^2_{\L(Y)}},
\end{gather}
where $\mathfrak{H}^j$  is a set of all $j$-dimensional subspaces in $\dom(\h\e^{\theta_p})$. 
Recall that   $\Psi:\C_l^{m+1}\to\L(Y)$ is defined by \eqref{Psi}. It is easy to see that
the image of $\Psi$ is contained in $\dom(\h\e^{\theta_p})$, and
\begin{gather}
\label{unitar}
\|\Psi\u\|_{\L(Y)}=\|\u\|_{\C_l^{m+1}},\quad 
\h\e^{\theta_p}\la \Psi\u,\Psi\u\ra=\h_0^N\la \u,\u\ra
\end{gather}
(recall, that the form $\h_0^N$ is given by \eqref{hN0}, by $\H_0^N$ we denote the associated operator).

Let $\{\mathbf{e}^1,\mathbf{e}^2,\dots,\mathbf{e}^{m+1}\}$ be an orthonormal system of eigenvectors of $\H_0^N$ such that $\H_0^N \mathbf{e}^j=
B_{j-1} \mathbf{e}^j$ (see \eqref{lambdaB}). For $j=1,\dots,m+1$ we set 
$\mathbf{W}^j:=\mathrm{span}(\mathbf{e}^1,\dots, \mathbf{e}^j)$. It is easy to see that
\begin{gather}
\label{max}
\max\limits_{\u\in \mathbf{W}^j\setminus\{0\}}
{\h_0^N\la \u,\u\ra\over \|\u\|^2_{\C_l^{m+1}}}= B_{j-1}.
\end{gather} 
Finally, we set $V^j:=\Psi\mathbf{W}^j$, obviously $V^j\in \mathfrak{H}^j$. Then using \eqref{minmax}-\eqref{max} we obtain: 
\begin{gather*}
\lambda_{j}(\H\e^N)\leq
\max_{u\in V_j\setminus\{0\}}{\h\e^{\theta_p}\la u,u\ra	\over \|u\|^2_{\L(Y)}}=
\max_{\u\in \mathbf{W}_j\setminus\{0\}}{\h_0^N\la \u,\u\ra	\over \|\u\|^2_{\C_l^{m+1}}}=
B_{j-1}.
\end{gather*}
The lemma is proven.  
\end{proof}
 
\subsection{Estimates on  $\lambda_{k} (\H^{\theta_a}\e)$ and $\lambda_{k} (\H^D\e)$}
 
Let $\C^{m}_l$ be the subspace of $\C^{m+1}$ consisting of vectors of the form
$\u=(0,u_1,\dots,u_m)$ with $u_j\in\C$ with the scalar product generated by \eqref{scalar-product}, i.e.
\begin{gather*}
(\u,\v)_{\C^{m}_l}=\suml_{j=1}^m u_j \overline{v_j} l_j.
\end{gather*}
In this space  we introduce the quadratic form
$$
\ds\h^{\theta_a}_0\la \u,\u\ra=\suml_{j=1}^m  
\al_j\be_j ^2 N_j \left| u_j\right|^2.
$$
It is easy to see that 
$\h^{\theta_a}_0=\h^{N}_0\restriction_{\C^{m}_l}.$
The operator associated with this form is given by the matrix
\begin{gather*}
\H_0^{\theta_a}=\mathrm{diag}\left(\al_1\be_1 ^2 N_1 l_1^{-1},\,\al_2\be_2 ^2 N_2 l_2^{-1},\,\dots,\,\al_m\be_m ^2 N_ml_m^{-1}\right).
\end{gather*} 
Evidently, the eigenvalues of this matrix are the numbers 
$A_1<A_2<\,\dots\,<A_m$.

\begin{lemma}\label{lmA}
There exist such constants  $C_A>0$ and $\eps_A>0$ that
\begin{gather}\label{lmA:2}
A_j\leq \lambda_j(\H\e^{\theta_a})+C_A\eps^{1/2},\ j=1,\dots,m\quad\text{as }\eps<\eps_A.
\end{gather}
\end{lemma}

\begin{proof}
The proof is   similar to the proof  of Lemma~\ref{lmN}. 
There is only one essential difference: instead of the operator  $\Phi$ \eqref{Phi} one should use the operator  
$\Phi_0:\dom(\h^{\theta_a}\e)\to \C_l^{m}$ defined by
\begin{gather*}	
(\Phi_0 u)_0=0,\quad (\Phi_0 u)_j=(\Phi u)_j,\ j=1,\dots, m.
\end{gather*}	
and, as a consequence, instead of the Poincar\'e inequality \eqref{poincare} on $Y_0$ 
one should use  the  inequality 
\begin{gather*}
\|u\|^2_{\L(Y_0)}\leq
C_1\suml_{e\in \E_{Y_0}}\|u_e'\|^2_{\L(0,l(e))}, 
\end{gather*}
where $C_1=(\lambda_1(\Hb_{0,0}^{\theta_a}))^{-1}$ (recall, that the operator $\Hb_{0,0}^{\theta}$ is introduced in the proof of Lemma~\ref{lm1}, and its first eigenvalue is non-zero provided $\theta\not=\theta_p$). 
\end{proof}
 
\begin{lemma}\label{lmD}
One has:
\begin{gather*}
 \lambda_j(\H\e^D)\le A_j,\ j=1,\dots,m.
\end{gather*}
\end{lemma}

\begin{proof}
The proof  is similar to the proof of Lemma~\ref{lmP}. Namely, 
one has to replace everywhere in the proof of Lemma~\ref{lmP} the supscript $\theta_p$ by $D$, 
 the supscript $N$ by $\theta_a$, $B_{j-1}$  by $A_j$,
and to use instead of the mapping $\Psi$ \eqref{Psi} 
its restriction   to $\C^{m}_l$ (the image of this restriction is contained in $\dom(\h\e^D)$).  
\end{proof}

\subsection{End of the proof of Theorem~\ref{th1}}
It follows from \eqref{repres1+}, \eqref{enclosure} and Lemmata~\ref{lmN}-\ref{lmP} that 
\begin{gather}
\label{left:edge}
B_{j-1}-C_B\eps^{1/2}\leq \inf (L_{j,\eps})\leq B_{j-1},\ j=1,\dots,m+1
\end{gather}
as $\eps<\eps_B$. Similarly, using  
\eqref{repres1+}, \eqref{enclosure} and   Lemmata~\ref{lmA}-\ref{lmD} we get 
\begin{gather}
\label{right:edge}
A_j-C_A\eps^{1/2}\leq \sup(L_{j,\eps})\leq A_j,\ j=1,\dots,m
\end{gather}
as $\eps<\eps_A$. Finally, we infer from \eqref{repres1+} and Lemma~\ref{lm1} that
\begin{gather}\label{last:edge}
\Lambda_0\eps^{-1}< 
\sup(L_{m+1,\eps})\quad\text{as }\eps<\eps_\Lambda.
\end{gather}
Combining \eqref{interlace}, \eqref{left:edge}-\eqref{last:edge}
we conclude that there exists such $\eps_0>0$   that properties \eqref{th1:spec}-\eqref{th1:estim} hold for $\eps<\eps_0$, $\Lambda_0$ being defined by \eqref{Lambda0}, $C_A$ being defined in Lemma~\ref{lmA}, $C_B$ 
being defined in Lemma~\ref{lmN}. Evidently, $\Lambda_0 $  depends only on $Y$, while $\eps_0 $, $C_A$, $C_B$ depend   also on
$\al_j$, $\be_j$, $\ga$.  
Theorem~\ref{th1} is proven.

\begin{remark}
The proof of Theorem~\ref{th1} relies, in particular, on some properties of the eigenvalues of the operator $\H^{\theta_a}\e$ -- see the estimates \eqref{lm1:ineq}, \eqref{lmA:2}.
In fact, the only specific property of $\theta_a$ we use is that 
$\theta_a\not=\theta_p$. Thus, instead of $\H^{\theta_a}\e$ one can utilize any other $\H^{\theta}\e$ with $\theta\not=\theta_p$ -- the above estimates are still valid for its eigenvalues (but, of course, with another constants $\Lambda$, $\eps_\Lambda$, $C_A$, $\eps_A$). 
\end{remark}

\section{\label{sec3}Control over the endpoints of spectral gaps}
 
Our first goal is to show that under a suitable choice of  coupling constants 
$\al_j$, $\be_j$, $\gamma$ the numbers $A_j,\,B_j$ (cf.~Theorem~\ref{th1})  coincide with prescribed ones.
 
Throughout this section we  will use the notation $\H\e[\al,\be,\ga]$ for the operator $\H\e$
defined in Subsection~\ref{subsec13} (recall that this operator is associated with the form given by \eqref{form}); here  $\al=(\al_1,\dots,\al_m)\in \R^m$, $\be=(\be_1,\dots,\be_m)\in \R^m$, $\ga\in\R$ be such that $\al_j\not=0$, $\be_j\not=0$ and, moreover, \eqref{Amono} holds (so, we are in the framework of Theorem~\ref{th1}).  For the numbers $A_j$  and $B_j$  defined by 
\eqref{Aj},\,\eqref{equat},\,\eqref{interlace}
we will use the notations $A_j[\al,\be,\ga]$ and $B_j[\al,\be,\ga]$, respectively.

\begin{theorem}
\label{th2}
Let $\A_j$, $j=1,\dots,m$ and $\B_j$, $j=0,\dots,m$
  be arbitrary numbers
satisfying
\begin{gather}
\label{interlace+}
\B_{0 }<\A_{1 }<\B_{1 }<\A_{2}<\B_{2}<\dots<\A_{m }<\B_{m },\quad
\A_j\not=0,\ j=1,\dots,m.
\end{gather}
We set
\begin{gather}\label{al:be:ga}
\ds
\wt\al_j={\wt r_j(\A_j-\B_0)l_0\over N_j},\qquad
\wt\be_j=
\sqrt{
{\A_j   l_j\over \wt r_j(\A_j-\B_0)l_0}},\qquad
\wt\gamma=\B_0\left(1+\suml_{j=1}^m \wt r_j\right)l_0,
\end{gather}
where  $\wt r_j$, $j=1,\dots,m$ is defined by
\begin{gather}\label{r}
\wt r_j={\B_j-\A_j\over \A_j}\prod\limits_{i=\overline{1,m}|i\not=
j} \left({\B_i-\A_j\over \A_i-\A_j}\right).
\end{gather}
Then  
\begin{gather*}
A_j[\wt\al_j,\wt\be_j,\wt\ga]=\wt A_j,\ j=1,\dots,m,\qquad
B_j[\wt\al_j,\wt\be_j,\wt\ga]=\wt B_j,\ j=0,\dots,m.
\end{gather*}
\end{theorem}

\begin{remark}
The quantity standing under the symbol of square root in \eqref{al:be:ga}  is indeed positive.
This follows easily from  \eqref{interlace+} (the crucial observation: one has
$\mathrm{sign}(\B_i-\A_j)=\mathrm{sign}(\A_i-\A_j)\not= 0$ as $i\not= j$).
\end{remark}

\begin{proof}[Proof of Theorem~\ref{th2}]
The equality $A_j[\wt\al_j,\wt\be_j,\wt\ga]=\A_j$, $j=1,\dots,m$ is straightforward -- one just needs to insert $\wt\al_j$ and $\wt\be_j$ defined by \eqref{al:be:ga} into the definition of the numbers $A_j[\wt\al_j,\wt\be_j,\wt\ga]$  \eqref{Aj}. 

Now, let us prove that $B_j[\wt\al_j,\wt\be_j,\wt\ga]=\B_j$ as $j=0,\dots,m$.
For this purpose, we consider the following system of linear algebraic equations
(for unknown $z=(z_1,\,z_2,\,,\dots,\,z_m)\in\mathbb{C}^m$):
\begin{gather*}
\suml_{i=1}^m {\A_i\over \A_i-\B_j}z_i=-1,\quad j=1,\dots,m.
\end{gather*}
It was proven in \cite{Kh12} that $z=(\wt r_1,\dots,\wt r_m)$ with $\wt r_j$ being defined by \eqref{r} is the solution to this system.
Thus for $j=1,\dots,m$ one has\ \
$\ds
\suml_{i=1}^m \A_i (\A_i-\B_j)^{-1}\wt r_i=-1$ 
or, equivalently,
\begin{gather}\label{system1}
\forall j\in\{0,\dots,m\}:\quad
(\B_j-\B_0)\suml_{i=1}^m {\A_i\over \A_i-\B_j}\wt r_i=\B_0-\B_j.
\end{gather}
It is straightforward to check that \eqref{system1} is equivalent to
\begin{gather}
\label{equat1}
\forall j\in\{0,\dots,m\}:\quad
\B_j\left(l_0+\suml_{j=1}^m {\A_il_i  \over \wt\be_i^2(\A_i-\B_j)}\right)=\wt\gamma.
\end{gather}
Using $A_j[\wt\al_j,\wt\be_j,\wt\ga]=\A_j$  we conclude
from \eqref{equat1} that $\B_j$, $j=0,\dots,m$ are the roots of \eqref{equat} in which $\al_j=\wt\al_j$, $\be_j=\wt\be_j$, $\gamma=\wt\gamma$ are set. Hence $\B_j=B_j[\wt\al_j,\wt\be_j,\wt\ga]$ as $j=0,\dots,m$. Theorem~\ref{th2} is proven.
\end{proof}

Theorems~\ref{th1},\,\ref{th2} yield
that $\sigma(\H\e[\wt\al,\wt\be,\wt\ga])$ has $m$ gaps within
$(-\infty,\Lambda_0\eps^{-1}]$ as $\eps<\eps_0$,
moreover the endpoints of these $m$ gaps and the bottom of the spectrum   converge  to prescribed numbers as $\eps\to 0$.
Our next goal is to improve this result: we show that under a proper choice of  $ \al_j$  one can ensure the precise coincidence the left endpoints of the spectral gaps of $\H\e[\al,\wt\be,\wt\ga]$ with prescribed numbers. 

\begin{theorem}
\label{th3}
Let $\A_j$, $j=1,\dots,m$ and $\B_j$, $j=0,\dots,m$
  be arbitrary numbers
satisfying \eqref{interlace+}, and
let  $\wt\be_j$, $\wt\ga$ be defined by \eqref{al:be:ga}.
Then   there exists such $\wt\eps>0$ and $  C_0>0$ that
\begin{gather*}
 \forall \eps<\wt\eps\quad \exists\al=\al(\eps)\in\mathbb{R}^m:\qquad
\sigma(\H\e[\al,\wt\be,\wt\ga])\cap (-\infty,\Lambda_0\eps^{-1}]=
[B_{0,\eps},\Lambda_0\eps^{-1}]\setminus\cupl_{j=1}^m (\wt A_{j},B_{j,\eps}),
\end{gather*}
where $\Lambda_0$ is defined by \eqref{Lambda0}, $B_{0,\eps}<\wt A_1<B_{1,\eps}<\wt A_2<B_{2,\eps}<\dots<\wt A_m<B_{m,\eps}<\Lambda_0\eps^{-1}$, moreover
$$0\le \wt B_j  - B_{j,\eps}\leq C_0\eps^{1/2},\ j=0,\dots,m.$$
\end{theorem} 

The proof of Theorem~\ref{th3} is based on the following multi-dimensional version
of the intermediate value theorem established in \cite{HKP97}.

\begin{lemma}\cite[Lemma~3.5]{HKP97}\label{lemma:HKP}
Let $\mathcal{D}=\Pi_{k=1}^m[a_k, b_k]$ with $a_k < b_k$, $k=1,\dots,m$, and suppose we
are given a continuous function $F:\mathcal{D}\to\R^m$ such that each component $F_k$ of $F$ is
monotonically increasing in each of its arguments. Let us suppose that
$F_k^-<F_k^+$, $i=1,\dots,m$, where
\begin{gather*}
F_k^-=F(b_1,b_2,\dots,b_{k-1},a_k,b_{k+1},\dots,b_m),\quad
F_k^+=F(a_1,a_2,\dots,a_{k-1},b_k,a_{k+1},\dots,a_m).
\end{gather*}
Then for any $F^*\in\Pi_{k=1}^m[F_k^-,F_k^+]$
there exists a point $x\in\mathcal{D}$ such that $F(x)=F^*$.
\end{lemma}

\begin{proof}[Proof of Theorem~\ref{th3}]
Let $\delta>0$ and  $\mathcal{D}:=\Pi_{k=1}^m[\wt\al_k-\delta,\wt\al_k+\delta]$, where 
$\wt\al_1,\dots,\wt\al_m$ be defined by \eqref{al:be:ga}.
We assume that $\delta$ is so small that 
\begin{gather}\label{Amono+}
\forall\al\in\mathcal{D}:\quad \al_j\not=0,\ j = 1, \dots , m \text{\quad and\quad }
A_j[\al,\wt\be,\wt\ga] < A_{j+1}[\al,\wt\be,\wt\ga],\ j = 1, \dots , m - 1.
\end{gather}
This could be indeed achieved since \eqref{Amono+} holds for $\al=\wt\al$. Thus  Theorem~\ref{th1} is applicable for each $\al\in\mathcal{D}$. 
Also, analyzing the proof of Theorem~\ref{th1}, it is easy to see that the constants
$\eps_0,\,C_0$ in Theorem~\ref{th1} can be chosen \emph{the same for  all} $\al\in\mathcal{D}$;
the proof of this fact relies on the compactness of $\mathcal{D}$.
Hence   there exists  
$\eps_0>0$ and $C_0>0$ such that  
\begin{gather}
\label{spectrum:structure}
\forall\eps<\eps_0\quad \forall \al\in\mathcal{D}:\qquad
\sigma(\H\e[\al,\wt\be,\wt\ga])\cap (-\infty,\Lambda_0\eps^{-1}]=
[B_{0,\eps},\Lambda_0\eps^{-1}]\setminus\cupl_{j=1}^m (A_{j,\eps},B_{j,\eps}),
\end{gather}
where $A_{j,\eps}$, $B_{j,\eps}$ satisfy \eqref{th1:inter} and \eqref{th1:estim} (with $A_j[\al,\wt\be,\wt\ga]$, $A_j[\al,\wt\be,\wt\ga]$ instead of $A_j$ and $B_j$).
Further, for 
these  $A_{j,\eps}$, $B_{j,\eps}$ we will use the notations $A_{j,\eps}[\al,\wt\be,\wt\ga]$, $B_{j,\eps}[\al,\wt\be,\wt\ga]$, respectively.
We denote 
$$\al_j^\pm:=\wt\al_j\pm\delta,\quad\al^\pm:=(\al_{1}^\mp,\al_{2}^\mp,\dots,\al_{j-1}^\mp,\al_j^\pm,\al_{j+1}^\mp,\dots,
\al_{m-1}^\mp,\al_{m}^\mp),\quad 
A_{j,\eps}^\pm:=A_{j,\eps}[\al^\pm,\wt\be,\wt\ga].
$$
It is easy to see that there exists such $\wt\eps\in(0,\eps_0]$  that 
\begin{gather}\label{HKP}
\forall\eps<\wt\eps :\quad  
A_{j,\eps}^-< \wt A_{j}< A_{j,\eps}^+,\ j=1,\dots,m.
\end{gather}
Indeed,  
since $\al_j^-<\wt\al_j<\al_j^+$ and $\wt A_j\overset{\text{Th.~\ref{th2}}}{=}A_j[\wt\al,\wt\be,\wt\ga]=\wt\al_j\wt\be_j^2 N_jl_j^{-1}$, then 
\begin{gather}\label{HKP2}
\forall j=1,\dots,m:\quad
A_{j}^-< \wt A_{j}< A_{j}^+,
\end{gather}
where $A_j^\pm:=A_j[\al^\pm,\wt\be,\wt\ga]=\al_j^\pm\wt\be_j^2 N_jl_j^{-1}$.
Moreover for $\eps<\eps_0$ we have 
\begin{gather}\label{HKP1}
0\leq A_j^\pm-A_{j,\eps}^\pm\le C_0\eps^{1/2}.
\end{gather}
Property \eqref{HKP} follows immediately from \eqref{HKP2}-\eqref{HKP1}.

Now, let us fix $\eps\in(0,\wt\eps]$.
We introduce  the  function $F=(F_1,\dots,F_m):\mathcal{D}\to \mathbb{R}^m$ by 
\begin{gather}\label{F}
\al\overset{F_k}\longmapsto A_{k,\eps}[\al,\wt\be,\wt\ga],\quad k=1,\dots,m.
\end{gather}
The functions $F_k$ are continuous. Indeed, let $\al,\al'\in\mathcal{D}$. 
To simplify the presentation we assume that $\al_j,\al'_j,\wt\gamma\ge 0$ (and consequently  $\H\e[\al,\wt\be,\wt\ga]\ge 0$, $\H\e[\al',\wt\be,\wt\ga]\ge 0$); general case need   slight modifications.
By virtue of \eqref{FRT} one has for $f,g\in\L(\Gamma)$,
\begin{multline}\label{res:diff}
\left((\H\e[\al,\wt\be,\wt\ga]+\mathrm{I})^{-1}f-(\H\e[\al',\wt\be,\wt\ga]+\mathrm{I})^{-1}f,g\right)_{\L(\Gamma)}=\h\e[\al',\wt\be,\wt\ga]\la u,w\ra-\h\e[\al,\wt\be,\wt\ga]\la u,w\ra\\=
\suml_{i\in\Z^n}\suml_{j=1}^m\suml_{v\in\V_{ij}}(\al'_j-\al_j)(u_0(v)-\wt\be_j u_j(v))
\overline{(w_0(v)-\wt\be_j w_j(v))},
\end{multline}
where $u= (\H\e[\al,\wt\be,\wt\ga]+\mathrm{I} )^{-1}f$, $w=(\H\e[\al',\wt\be,\wt\ga]+\mathrm{I} )^{-1}g$, $\h\e[\al,\wt\be,\wt\ga]$ is a form associated with $\H\e[\al,\wt\be,\wt\ga]$. Using   \eqref{trace} and taking into account that
$\al_j,\al_j',\wt\ga\ge 0$
we  continue \eqref{res:diff} as follows,
\begin{multline}
\label{res:diff:est}
\left|\left((\H\e[\al,\wt\be,\wt\ga]+\mathrm{I})^{-1}f-(\H\e[\al',\wt\be,\wt\ga]+\mathrm{I})^{-1}f,g\right)_{\L(\Gamma)}\right|\leq 
C|\al-\al'|\|u\|_{\W^1(\Gamma)}\|w\|_{\W^1(\Gamma)}\\\leq 
C\eps^{1/2}|\al-\al'|
\left(\h\e[\al,\wt\be,\wt\ga]\la u, u\ra+\|u\|^2_{\L(\Gamma)}\right)^{1/2}
\left(\h\e[\al',\wt\be,\wt\ga]\la w, w\ra+\|w\|^2_{\L(\Gamma)}\right)^{1/2}\\=
C\eps^{1/2}|\al-\al'|\sqrt{(f,u)_{\L(Y)}(g,w)_{\L(Y)}}\leq 
C\eps^{1/2}|\al-\al'|\|f\|_{\L(Y)}\|g\|_{\L(Y)},
\end{multline}
where $C>0$ is a constant.
It follows from \eqref{res:diff:est} that 
$$
\left\|
(\H\e[\al,\wt\be,\wt\ga]+\mathrm{I})^{-1}-
(\H\e[\al',\wt\be,\wt\ga]+\mathrm{I})^{-1}
\right\|\to 0
\text{ as }\al-\al'\to 0,
$$
whence for an arbitrary compact set $\mathcal{I}\subset\R$ one has
\begin{gather}\label{Hausdorff}
\mathrm{dist}_{\rm H}(\sigma(\H\e[\al,\wt\be,\wt\ga])\cap\mathcal{I},\sigma(\H\e[\al',\wt\be,\wt\ga])\cap\mathcal{I})\to 0\text{ as }\al-\al'\to 0,
\end{gather}
where $\mathrm{dist}_{\rm H}(\cdot,\cdot)$ stands for the Hausdorff distance.
Taking into account a special structure of $\sigma(\H\e[\al,\wt\be,\wt\ga])$   \eqref{spectrum:structure}  we conclude from \eqref{Hausdorff} that
$A_{k,\eps}[\al',\wt\be,\wt\ga]- A_{k,\eps}[\al,\wt\be,\wt\ga]\to 0\text{ as }\al-\al'\to 0,$ i.e.
  $F_k$ is continuous.
It is clear that $A_{k,\eps}[\al,\wt\be,\wt\ga]$ is the right endpoint  of 
the $k$th spectral band:
\begin{gather}\label{max1}
A_{k,\eps}[\al,\wt\be,\wt\ga]=\max_{\theta\in\mathbb{T}^n}\lambda_k(\H\e^\theta[\al,\wt\be,\wt\ga]),
\end{gather}
where $\H\e^\theta[\al,\wt\be,\wt\ga]$ denotes the operator $\H\e^\theta$ with $\be_j=\wt\be_j$, $\ga=\wt\ga$.
Since $\H\e^\theta[\al,\wt\be,\wt\ga]\leq \H\e^\theta[\al',\wt\be,\wt\ga]$ (in the form sense) as $\al_j\le \al'_j$, $\forall  j=1,\dots,m$, by min-max principle
we conclude for $k=1,\dots,m$:
\begin{gather}
\label{la:la}
\forall\theta\in\mathbb{T}^n:\
\lambda_k(\H\e^\theta[\al,\wt\be,\wt\ga])\leq \lambda_k(\H\e^\theta[\al',\wt\be,\wt\ga])\text{ provided }\al_j\le \al'_j,\ \forall j=1,\dots,m.
\end{gather}
It follows from \eqref{max1}-\eqref{la:la} that the functions $F_k$ increase monotonically in each of their arguments. 
Taking into account \eqref{HKP} we infer that the function $F$ satisfy all the requirements
of Lemma~\ref{lemma:HKP}. Applying this lemma we conclude that there exists such $\al\in\mathcal{D}$ that
\begin{gather}\label{f=A}
F_k(\al)=\wt A_k,\ k=1,\dots,m.
\end{gather}
Combining \eqref{spectrum:structure},   \eqref{F}, \eqref{f=A} we arrive at the statement of Theorem~\ref{th3}. 
\end{proof}

\begin{remark} 
The assumption $\A_j\not=0$, $j=1,\dots,m$ in \eqref{interlace+} is  essential -- one cannot avoid it when using 
the Hamiltonians $\H\e$ introduced in Subsection~\ref{subsec13}, since the numbers $A_j$   \eqref{Aj} are always non-zero. To overcome this restriction one can add to $\H\e$ a constant potential, which shift the spectrum accordingly. Another option is to
to pick in each $Y_j$, $j=0,\dots,m$ an internal point $\widehat v_j$, and then to add at $\widehat v_j$ the $\delta$-coupling of the strength $\widehat\gamma\, l_{j}$, where $l_j$ is defined by \eqref{lj} and  $\widehat\gamma\in\mathbb{R}$. Denote by $\widehat{\mathcal{H}}\e$ the modified Hamiltonian.
Repeating verbatim the arguments we use in the proof of Theorem~\ref{th1} one can show that the spectrum of $\widehat{\mathcal{H}}\e$ satisfies  \eqref{th1:spec}-\eqref{th1:estim}, but with $A_{j}+\widehat\gamma$ and $B_{j}+\widehat\gamma$
instead of $A_{j}$ and $B_j$.
\end{remark}

\section*{Acknowledgment}

The author is supported by  Austrian Science Fund (FWF) under the Project M~2310-N32.

\end{document}